\documentclass[a4paper,11pt]{article}
\usepackage{amsmath,amsthm,amssymb,bm,array,indentfirst}
\usepackage[inline]{enumitem}
\usepackage[
top=2.7cm, bottom=3.2cm,
outer=3cm, inner=3cm,
headheight=0cm, headsep=0cm,
footskip=1cm]{geometry}
\renewenvironment{abstract}
{\list{}{\rightmargin=2.6cm\leftmargin=\rightmargin}%
\item[]{\sc Abstract.}\hspace{-0.5em}
\small % size of the text
\relax}{\endlist}
\usepackage{fancyhdr} 
\theoremstyle{custom} 
\newtheorem{thm}{Theorem}
\newtheorem{lem}{Lemma}
\newtheorem{prop}{Proposition}

\newtheorem{rmk}{Remark}
\newtheorem*{ac*}{Acknowledgement}
\newtheorem*{rmk*}{Remark}
\DeclareRobustCommand{\rchi}{{\mathpalette\irchi\relax}}
\newcommand{\irchi}[2]{\raisebox{0.5\depth}{$#1\chi$}}% to raise '\chi'
\usepackage[marginal,hang]{footmisc}
\renewcommand\footnoterule{{\hrule width 4cm height .7pt \kern 0.5ex}}% footnote rule
\makeatletter % un-numbered footnote
\def\nfootnote{\def\@thefnmark{}\@footnotetext} 
\makeatother
\usepackage{hyperref}

\begin{document}

\nfootnote{MSC 2000: 22E43, 11F03, 11F70}

\title{\bf Upper bounds for geodesic periods over hyperbolic manifolds}
\author{Feng~~Su}
\date{}
\maketitle
\thispagestyle{empty}
\allowdisplaybreaks
\vspace*{-0.5cm}

\begin{abstract}
We prove an upper bound for geodesic periods of Maass forms over hyperbolic manifolds. By definition, such periods are integrals of Maass forms restricted to a special geodesic cycle of the ambient manifold, against a Maass form on the cycle. Under certain restrictions, the bound will be uniform.
\end{abstract}

\medskip

\section{Introduction}\label{intro}
Let $X$ be a $d$-dimensional connected complete hyperbolic manifold with finite volume, $\phi$ a square integrable Laplace eigenfunction on $X$, whose eigenvalue is denoted by $\lambda\geq 0$. In the theory of automorphic forms, $\phi$ is also called  ``Maass  form'' (after H. Maass \cite{ma}).  We normalize $\phi$ so that it has $L^2$-norm $1$. Let $Y$ be a special cycle of $X$ 
which is compact, totally geodesic and has  codimension 1 (see Section \ref{rep} for the precise description). Fix a hyperbolic measure $dy$ on $Y$.
Given a normalized Maass form $\psi$ on $Y$ with the Laplace eigenvalue 
$\mu\geq 0$, define
the period integral 
$$P_Y(\phi,\psi):=\int_Y\phi(y)\psi(y)dy.$$ 
This integral converges since $Y$ is compact and $\phi$, $\psi$ are smooth (by the elliptic regularity theorem). We call $P_Y(\phi,\psi)$ {\it geodesic period} from the geometric  perspective. Such a period fits into the general notion of {\it automorphic period} which plays a central role in the study of automorphic forms thanks to its close relations with automorphic representations and special values of certain automorphic $L$-functions 
(see \cite{ggp}, \cite{ii}, \cite{wa}, \cite{z} and references therein).

The aim of this paper is to prove an upper bound for the geodesic period $P_Y(\phi,\psi)$. 
\begin{thm}\label{thm}
There exists a positive constant $C_{\psi}$ 
depending on $\psi$ such that 
$$\big|P_Y(\phi,\psi)\big|\leqslant C_{\psi},
\quad\text{as}\,~\lambda\rightarrow\infty.$$ 
\end{thm}

We briefly review the history on the bounds for geodesic periods. Let us first focus on the case where $X$ is a hyperbolic surface. In this case the compact geodesic cycle $Y$ with codimension 1 is just a closed geodesic and we may choose $\psi$ to be a unitary character of $Y\simeq S^1$. 
In particular, when $X$ is compact,
the boundedness of 
$P_Y(\phi,1):=\int_Y\phi(y)dy$ 
was proved in \cite{go, he} 
by using the trace formula; 
in \cite{cs} the PDE techniques was used to show that $P_{Y}(\phi,1)$ 
converges to $0$ as $\lambda$ tends to 
$\infty$ (logarithm decay). 
See also \cite{bo, bgt} and a 
series of papers by Sogge and his 
collaborators for various work on 
$P_{Y}(\phi,1)$ in this setting.  
When $X$ is not necessarily compact and general unitary characters $\psi$ are involved, Reznikov \cite{re} proved that $P_Y(\phi,\psi)$ is bounded (from above) by a constant. As for the higher dimensional situation, a very general result by Zelditch \cite{ze} says that $P_Y(\phi,1)$ is bounded by a constant, for any compact Riemannian manifold $X$ without boundary and its submanifold $Y$ (see Remark 3.5 loc.\,cit.). 
However, due to Lemma 3.1 loc.\,cit.~where the Lagrangean distribution (from which the period is constructed) should be positive, the function $\psi$ has to be a constant.

Our theorem is a natural extension of \cite{re} to the higher dimensional situation. 
In this paper, $X$ is not necessarily compact, and we shall deal with general $\psi$.

Another type of question is to estimate $P_Y(\phi,\psi)$ where $\phi$ is fixed and $\psi$ varies. 
It turns out that the period decays exponentially with respect to $\mu$ (see \cite{mo}).
We do not expect such a strong decay in our setting. 
In fact, Reznikov \cite{re} has conjectured that $|P_Y(\phi,\psi)|\ll_\varepsilon\lambda^{-1/4+\varepsilon}$ over hyperbolic surfaces (when $\psi$ is fixed). 
This conjecture is, by Waldspurger's formula, consistent with the Lindel\"of conjecture on the critical value of certain automorphic $L$-functions. 
For $n$-dimensional compact special geodesic cycle $Y$ of $X$, we conjecture the following bound: 
$|P_Y(\phi,\psi)|\ll_{\varepsilon,\psi}\lambda^{-n/4+\varepsilon}$ 
(see Remark \ref{rmk3}). 

As in \cite{re}, the bound in Theorem \ref{thm} is not uniform with respect to $\psi$. 
In Section \ref{pr2}, we shall explain why it is not and then refine the proof of Theorem \ref{thm} to get a uniform bound, when $\phi$ and $\psi$ are under certain
restrictions.

Write $\lambda=\big(\frac{d-1}{2}\big)^2-\tau^2$ and $\mu=\big(\frac{d-2}{2}\big)^2-\tau^{\prime\,2}$. When $\lambda$, $\mu$ are large, $\tau$, $\tau'$ are purely imaginary.
\begin{thm}\label{thm2} 
\begin{enumerate}[label={\upshape(\arabic*)},leftmargin=*,topsep=.5ex,itemsep=-.5ex]
\item\label{c1}
If $|\tau|$, $|\tau'|$ are large and $\big||\tau|-|\tau'|\big|<D$ where $D$ is any fixed positive number, then 
$$|P_Y(\phi,\psi)|\ll_{_D} |\tau|^{\frac{d}{2}-\frac34}.$$  
\item\label{c2}
If $|\tau|$, $|\tau'|$ are large,  $\big||\tau|-|\tau'|\big|$ is not bounded and $|\tau|>|\tau'|$, then 
$$|P_Y(\phi,\psi)|\ll\frac{|\tau|^{\frac12}\cdot|\tau'|^{\frac{d-2}{2}}}{\big(|\tau|^2-|\tau'|^2\big)^{\frac14}}.$$  
\item   
If $|\tau|$, $|\tau'|$ are large,  $\big||\tau|-|\tau'|\big|$ is not bounded and $|\tau|<|\tau'|<C\cdot|\tau|$ where $C>1$ is any fixed positive number, then 
$$|P_Y(\phi,\psi)|\ll_{_C}\frac{|\tau|^{\frac12}\cdot|\tau'|^{\frac{d-2}{2}}}{\big(|\tau'|^2-|\tau|^2\big)^{\frac14}}\cdot e^{-\frac{\pi}{2}(|\tau'|-|\tau|)}.$$  
\end{enumerate} 
\end{thm}

\begin{rmk}
\begin{enumerate}[label={\normalfont(\roman*)}, wide=0pt, itemsep=-.1em]
\normalfont
\item
In \cite{mo}, the $O$-constant for the rapid decay of $P_Y(\phi,\psi)$ (as $\mu$ goes to $\infty$) depends on $\phi$. To the author's knowledge, we do not have a uniform bound when $\mu$ and $\lambda$ are both large and (absolutely) independent of each other. Unlike the third case of Theorem \ref{thm2}, when $|\tau|\geq |\tau'|$ and $|\tau|$, $|\tau'|$ are both large we do not expect the exponential factor to appear in the bound.  
\item
When $Y$ is noncompact, $P_Y(\phi,\psi)$ still converges as long as $\phi$ decays fast at the cusp(s) of $Y$. In this case, Theorem \ref{thm} and \ref{thm2} are still valid by the standard truncation method.  
\item\label{rmk3}
Replacing $X$ with any finite-volume hyperbolic space, Theorem \ref{thm} and \ref{thm2} still hold (see Section \ref{rep}). 
\end{enumerate}
\end{rmk}

The method of this paper, an adoption of \cite{re}, is of representation-theoretic nature. We first express periods as special values of an automorphic functional. By the uniqueness of invariant functionals of automorphic representations, we shall reduce the study of periods to that of a special functional and the corresponding proportionality scalar. The value of the special functional can be explicitly computed. We shall bound the proportionality scalar via a geometric argument based on Fourier analysis.

The paper is organized as follows. In Section \ref{rep} we collect some facts about the structure and representation theory of the Lorentz group which will be used later. In Section \ref{pf} we prove Theorem \ref{thm}. In particular, we shall apply the trick ``fattening'' in \cite{re} to the higher dimensional situation. The Multiplicity One Theorem and intertwining operator that we shall need in this part have been established by Kobayashi--Speh \cite{ks}. In Section \ref{pr2}, we prove Theorem \ref{thm2}.

\begin{ac*}
The author is very grateful to Andre Reznikov for his kind support and valuable discussions. The author thanks Binyong Sun and the anonymous referees for their useful suggestions which help to improve the paper. The research was partially supported by the ISF grant 533/14.
\end{ac*}

\section{Preliminaries}\label{rep}
By the uniformization theorem (see, e.g. \cite{fj} in which the manifolds are assumed to be orientable), any $d$-dimensional connected complete hyperbolic 
manifold $X$ with finite volume (not necessarily orientable) is isometrically isomorphic to the 
locally symmetric space $\Gamma\backslash G/K$ 
where $G=O(1,d)$, $\Gamma\cong\pi_1(X)$ is a lattice of 
$G$, and $K\subset G$ is a maximal compact subgroup. We may and will choose $K$ to be $\{\textup{diag}(\pm1,k)|k\in O(d)\}$.

As $X$ is smooth, $\Gamma$ is torsion-free. In this paper one can also start with $\Gamma\backslash G$, instead of $X$. Then we do not have to assume that $\Gamma$ is torsion-free. Meanwhile, the method of the paper still works.   Hence, Theorem \ref{thm} and \ref{thm2} hold for any lattice $\Gamma\subset G$.

The quotient space $\Gamma\backslash G$ is equipped with a $G$-invariant Radon measure that inherits from a Haar measure of $G$ and descends to the hyperbolic measure of $\Gamma\backslash G/K$. The group $G$ acts on $L^2(\Gamma\backslash G)$ by the right 
regular translation $R$. 
From this action we get a unitary representation of $G$ on $L^2(\Gamma\backslash G)$. 
In view of the isomorphisim $L^2(\Gamma\backslash G/K)
\cong L^2(\Gamma\backslash G)^K$, we can lift 
$\phi\in L^2(\Gamma\backslash G/K)$ to $\tilde{\phi}\in L^2(\Gamma\backslash G)$ such that $\tilde{\phi}$ is $K$-fixed. 
Let $V_{\tilde{\phi}}$ denote the closed subspace in $L^2(\Gamma\backslash G)$ which is spanned by $\{R(g)\tilde{\phi}\,|\,g\in G\}$. The Killing form of the Lie algebra $\mathfrak{g}$ of $G$ defines the Casimir operator $\square$ over $C_c^{\infty}(\Gamma\backslash G)$ which extends to $L^2(\Gamma\backslash G)$ as a self-adjoint operator.  When restricted to the smooth vectors in $L^2(\Gamma\backslash G)^K$, $\square$ is identified with $\Delta$ (the hyperbolic Laplacian). The Casimir operator commutes with the action $R$. Hence the completed subspace $V_{\lambda}\subset L^2(\Gamma\backslash G)$ of Casimir eigenfunctions with eigenvalue $\lambda$ is a subrepresentation space of $G$. 
It is known that $V_{\lambda}^K$ is finite dimensional (say, $q$-dimensional) and $V_{\lambda}$ can be decomposed as a direct sum of $q$ copies of equivalent  subrepresentations:  $V_{\lambda}\cong q\pi_{\lambda}$ where $\pi_{\lambda}\subset L^2(\Gamma\backslash 
G)$ is a unitary irreducible spherical representation of $G$ and $\pi_{\lambda}\cong V_{\tilde{\phi}}$. This  results from the duality theorem in \cite{ge}. 
Conversely, each element in $\pi_{\lambda}^K$ leads to a Maass form in $L^2(\Gamma\backslash G/K)$.

The spherical unitary dual of $G$ had been determined. Any nontrivil element in the dual is $G$-equivalent to an induced representation ${\rm Ind}_{MAN}^{G}(\bm{1}\otimes e^{\tau}\otimes\bm{1})$ 
where $MAN$ is a 
minimal parabolic 
subgroup of $G$ and 
$\tau$ is an element in 
$\mathfrak{a}_{\mathbb{C}}^{\ast}$. 
Denote by $\mathcal{B}$ the isomorphisim $\pi_{\lambda}\cong{\rm Ind}_{MAN}^{G}(\bm{1}\otimes e^{\tau}\otimes\bm{1})$. 
Next we recall the notions which are used to define the induced representation (they will also be used later). 
Let $\Theta$ be the Cartan involution of $G$ given by taking the transpose inverse. The Cartan involution $\theta$ on Lie algebra level leads to the vector space decomposition of the Lie algebra $\mathfrak{g}$ of $G$, namely, $\mathfrak{g}=
\mathfrak{p}\oplus\mathfrak{k}$ where $\mathfrak{k}=\textup{Lie}(K)=\{X\in\mathfrak{g}\,|\,\theta(X)=X\}$ is a Lie subalgebra of $\mathfrak{g}$ and $\mathfrak{p}=\{X\in\mathfrak{g}\,|\,\theta(X)=-X\}$.  
Let $\mathfrak{a}$ be a maximal abelian subspace of 
$\mathfrak{p}$. For any linear functional $\alpha$ on $\mathfrak{a}$, 
define $\mathfrak{g}_{\alpha}=\left\{X\in\mathfrak{g}\,|\,[H,X]=\alpha(H)X~\textup{for all}~H\in\mathfrak{a}\right\}$.
Those nonzero $\alpha$ 
such that $\mathfrak{g}_{\alpha}\ne\{0\}$ constitute a root system, denoted $(\mathfrak{g},\mathfrak{a})$. Let $E_{ij}=(e_{ij})$ be a $(d+1)\times(d+1)$ matrix whose entries satisfy $e_{lk}=1$ for $(l,k)=(i,j)$, and $e_{ik}=0$ otherwise. 
Denote 
$$E=E_{12}+E_{21}\quad\textup{and}\quad E_i=E_{1(i+2)}+E_{2(i+2)}+E_{(i+2)1}-E_{(i+1)2}\quad
(1\leqslant i\leqslant d-1).$$
We may choose $\mathfrak{a}=
\mathbb{R}\,E$. Then the root system $(\mathfrak{g},\mathfrak{a})$ 
consists of two elements $\pm\alpha_0$ where $\alpha_0$ (the positive root) 
is defined by ${\rm ad}(E)$. 
As ${\rm ad}(E)E_i=E_i$ ($1\leq i\leq d-1$), we have $\alpha_0(E)=1$ and $\mathfrak{g}_{\alpha_0}
=\mathfrak{n}:=
\mathbb{R}E_1
\oplus\cdots\oplus
\mathbb{R}E_{d-1}$. 
Let $A=\exp(\mathfrak{a})$, $N=\exp(\mathfrak{n})$. 
The groups $A$ and $N$ are both abelian.
We have the Iwasawa decomposition $G=NAK$.
The centralizer of $A$ in $K$, denoted $M$, is 
isomorphic to $O(d-1)\times\mathbb{Z}/2\mathbb{Z}$ via the map $$\iota_{_\flat}:\,O(d-1)\rightarrow M,\quad k\mapsto\textup{diag}(\flat,\flat,k),$$
where $\flat=\pm1$.
Since $\mathfrak{a}$ is of dimension $1$, we may identify $\mathfrak{a}_{\mathbb{C}}^{\ast}$ with $\mathbb{C}$ via the map $\mathfrak{a}_{\mathbb{C}}^{\ast}\rightarrow\mathbb{C}$, $\alpha\mapsto \alpha(E)$. From now on, we shall not distinguish a root from its image in $\mathbb{C}$ under this map.
The half sum of positive roots (whose multiplicities are counted) in the root system $(\mathfrak{g},\mathfrak{a})$ is $\rho=\frac{d-1}{2}\alpha_0(E)=\frac{d-1}{2}$.
By definition, the induced representation is 
$$I(\tau)={\rm Ind}_{MAN}^{G}(\bm{1}\otimes e^{\tau}\otimes\bm{1})
:=\left\{f:\,G\rightarrow\mathbb{C}\,\big|\,f(gman)=a^{-(\tau+\rho)}f(g),~f|_K\in L^2(K)\right\}$$ with the $G$-action: $L(g)f(x)=f(g^{-1}x)$. 
It is known that $I(\tau)$ is irreducible and unitarizable 
if and only if $\tau\in(-\rho,\rho)\cup \textup{\bf i}\,\mathbb{R}$. 
For such $\tau$ we have $I(\tau)\cong I(-\tau)$. 
The trivial representation of $G$ 
is equivalent to a subrepresentation of $I(-\rho)$, and also equivalent 
to the Langlands quotient representation $L(\rho)$ of $I(\rho)$ 
(note that $I(\rho)$ is reducible and has a unique nontrivial subrepresentation).  
The eigenvalue of $\square$ on $I(\tau)$ is $\lambda=\rho^2-\tau^2$, 
equal to the Laplace eigenvalue of $\phi\in V_{\lambda}^K$.

The notion of {\it noncompact picture} is important to us in this paper. Here we give an introduction (see Chapter VII of \cite{kn}, or Section 2.3 of \cite{mo}). Denote $\bar{N}=N^{\tiny\mathrm{T}}$ (\,$^{\tiny\mathrm{T}}$ means transpose). In the Bruhat decomposition 
$$
G=MAN\cup\bar{N}MAN,
$$
$\bar{N}MAN$ is open and dense in $G$, so $F\in I(\tau)^\infty$ (the smooth part of $I(\tau)$) is 
completely decided by its restriction to $\bar{N}$. 
For $x=(x_1,\ldots,x_{d-1})\in\mathbb{R}^{d-1}$, denote $$n_x=\exp\left(\sum\limits_{i=1}^{d-1}x_iE_i\right)\quad\textup{and}\quad \bar{n}_x=
n_x^{\tiny\mathrm{T}}=\exp\left(\sum\limits_{i=1}^{d-1}x_iE_i^{\tiny\mathrm{T}}\right).$$ 
Given $F\in I(\tau)^{\infty}$, define $\mathcal{A}F\in C^{\infty}(\mathbb{R}^{d-1})$ by
$$(\mathcal{A}F)(x)=F\left(\bar{n}_x\right).$$
Denote by $J(\tau)$ the image of $I(\tau)$ under the map $\mathcal{A}$. Then $C^{\infty}_c(\mathbb{R}^{d-1})\subseteq J(\tau)$. The action of $G$ on $J(\tau)$ is given by 
$$g\bm{.}(\mathcal{A}F)=\mathcal{A}\big(L(g)F\big).$$
Write $g$ as $g=\bar{n}(g)m(g)a(g)n(g)$ according to the Bruhat decomposition.
The group $G$ also acts on $\mathbb{R}^{d-1}$. The action, denoted $g\bm{.}x$ (where $g\in G$, $x\in\mathbb{R}^{d-1}$), satisfies
$$
\bar{n}_{g\bm{.}x}=\bar{n}(g\bar{n}_x).$$
The action of $G$ on $J(\tau)$ can then be rephrased as
$$(g\bm{.}f)(x)=a\left(g^{-1}\bar{n}_x\right)^{-(\tau+\rho)}f\left(g^{-1}\bm{.}\,x\right),\quad f\in J(\tau).
$$
Note that the action of $G$ on $J(\tau)$ depends on the parameter $\tau$.
For $\tau\in \textup{\bf i}\,\mathbb{R}$, the invariant Hermitian form on $J(\tau)$ is 
$$\|f\|^2_{J(\tau)}=\frac{\Gamma(2\rho)}{\pi^{\rho}\Gamma(\rho)}\int_{\mathbb{R}^{d-1}}|f(x)|^2dx,\quad f\in J(\tau).$$
For $\tau\in(0,\rho)$, the invariant Hermitian form on $J(\tau)$ is 
$$\|f\|^2_{J(\tau)}=\frac{\Gamma(2\rho)\Gamma(\tau+\rho)}{\pi^{\rho}\Gamma(\rho)\Gamma(\tau)}\int_{\mathbb{R}^{d-1}}
\int_{\mathbb{R}^{d-1}}|x-y|^{2(\tau-\rho)}f(x)\overline{f(y)}dxdy,\quad f\in J(\tau).$$
Modulo a unit scalar, the element in $J(\tau)$ that corresponds to the normalized Maass  form $\phi$ is 
$$\mathfrak{f}(x)={\cal A}\circ
\mathcal{B}(\tilde{\phi})=(1+|x|^2)^{-(\tau+\rho)}$$ 
which has $L^2$-norm $1$ with respect to the above Hermitian forms.

There is a unique  proper subrepresentation $J_0(\rho)$ in $J(\rho)$.
Each element in $J_0(\rho)$ 
is annihilated by an intertwining operator (see Section 10 and 11, Chapter VII of \cite{kn}). In particular,
$f\in J_0(\rho)$ satisfies $\int_{\mathbb{R}^{d-1}}
f(x)dx=0$. Thus the form
$$\|f\|^2_{J(\rho)}=
\frac{\Gamma(2\rho)^2}{\pi^{\rho}
\Gamma(\rho)^2}
\left|
\int_{\mathbb{R}^{d-1}}
f(x)dx\right|^2,
\quad f\in J(\rho),$$
is well-defined on the Langlands quotient $\tilde{J}(\rho):=J(\rho)/J_0(\rho)$. Since $\tilde{J}(\rho)$
is one-dimensional (it is equivalent to the trivial representation)
and $\|\mathfrak{f}\|^2_{J(\rho)}=1$,
the above form 
$\|\cdot\|^2_{J(\rho)}$ 
is indeed a Hermitian 
form 
on $\tilde{J}(\rho)$.

Now we introduce the special cycle. Denote
$$G'=\{\textup{diag}(g_1,g_2)\in G\,|\,g_1\in O(1,d-1),~g_2=\pm1\}.$$
Then
$$K'=K\cap G'=\{\textup{diag}(k_1,k_2)\in K\,|\,k_1\in O(d-1),~k_2=\pm1\}$$
is a maximal subgroup of $G'$. 
Passing to a finite cover if necessary, the $(d-1)$-dimensional compact  special geodesic cycle of $X$ is realized as the image of the quotient $\Gamma'\backslash G'/K'$ embedded in $\Gamma\backslash G/K$
where the image is a totally geodesic submanifold of $X$. 
Here $\Gamma'=\Gamma\cap G'\cong\pi_1(Y)$ is a uniform lattice of $G'$. Likewise, the Laplace eigenvalue $\mu$ of $\psi$ can be written as $\mu=\rho'^2-\tau'^{2}$ where $\rho'=\frac{d-2}{2}$ is the half sum of positive roots of the root system
$(\mathfrak{g}_0,\mathfrak{a})$ and $\tau'\in[-\rho',\rho']\cup \textup{\bf i}\,\mathbb{R}$. 
Without loss of generality, we identify $Y$ with $\Gamma'\backslash G'/K'$. 
We refer the reader to \cite{sch} for applications of geodesic cycles to the cohomology of arithmetic groups and automorphic forms.

Denote by $J'(\tau')$ the noncompact picture of the induced representation $I'(\tau')$ of $G'$ where $I'(\tau')={\rm Ind}_{M'A'N'}^{G'}(\bm{1}\otimes e^{\tau'}\otimes\bm{1})$ with $M'=M\cap G'$, $A'=A\cap G'=A$, $N'=N\cap G'$. 
We use $J'(\cdot)$ and $I'(\cdot)$ to distinguish  representations of $G'$ from representations $J(\cdot)$ and $I(\cdot)$ of $G$.

\section{Proof of Theorem \ref{thm}} \label{pf}
This section is devoted to the proof  of Theorem \ref{thm}.  
Write $x=\left(x',x''\right)\in\mathbb{R}^{d-2}\times
\mathbb{R}$ for $x\in\mathbb{R}^{d-1}$. From now on, we shall always assume that $\tau\in \textup{\bf i}\,\mathbb{R}$ since $\lambda=\left(\frac{d-1}{2}\right)^2-\tau^2$ is assumed to be sufficiently large in this paper.

By \cite{ks}, the intertwining operators in ${\rm Hom}_{G'}(J(\tau),J'(\tau'))$ exist and are unique up to scalar multiples, provided that
\begin{equation}\label{mult}
(\tau+\rho,\tau'+\rho')\in\mathbb{C}^2\smallsetminus L_{\rm even},
\end{equation}
where $L_{\rm even}=\big\{(-i,-j):j\leq i,~j\equiv i~(\rm mod~2)\big\}$. This is the so-called Multiplicity One Theorem. All intertwining operators were classified in \cite{ks}. Note that the variables $\lambda$, $\nu$ in \cite{ks}, when adapted to the present paper, are equal to $\tau+\rho$ and $\tau'+\rho'$, respectively. 
Since ${\rm Hom}_{G'}\big(J(\tau),J'(\tau')\big)\cong
{\rm Hom}_{\Delta(G')}\big(J(\tau)\otimes J'(-\tau'),\mathbb{C}\big)$, we may translate the intertwining operator in ${\rm Hom}_{G'}\big(J(\tau),J'(\tau')\big)$ to be a $\Delta(G')$-bilinear form on $J(\tau)\times J'(\tau')$. In such a way, we get a $\Delta(G')$-bilinear form $L^{\rm mod}_{\tau,\tau'}$: 
$$L^{\rm mod}_{\tau,\tau'}:\,J(\tau)\times J'(\tau')\to\mathbb{C},\quad (f,h)\mapsto\int_{\mathbb{R}^{d-1}}\int_{\mathbb{R}^{d-2}}K_{\tau,\tau'}(x,y)f(x)h(y)dxdy,$$
where the kernel $K_{\tau,\tau'}$ is given by
$$K_{\tau,\tau'}(x,y)=\big(
\big|x'-y\big|^2+\big|x''\big|^2\big)^{\tau'-\rho'}\cdot
\big|x''\big|^{\tau-\rho-(\tau'-\rho')}.$$  

Denote the first term in $K_{\tau,\tau'}$ by $\mathcal{S}$, and the second term by $\mathcal{T}$.
The form $L^{\rm mod}_{\tau,\tau'}$ is well-defined on $J(\tau)\times \tilde{J}'(\rho')$ where $\tilde{J}'(\rho')$ stands for the Langlands quotient of $J'(\rho')$. The reason is as follows. When $\tau'=\rho'$ and $d\geq 3$, the kernel is simplified to be $|x''|^{\tau-\rho}$. Thus, we have
$$L^{\rm mod}_{\tau,\rho'}(f,h)=\int_{\mathbb{R}^{d-2}}h(y)dy\cdot\int_{\mathbb{R}^{d-1}}|x''|^{\tau-\rho}f(x)dx.$$
Recall that $\int_{\mathbb{R}^{d-2}}h(y)dy=0$ for any $h$ in the (unique) proper subrepresentation $J'_0(\rho')$ of $J'(\rho')$ (see Section \ref{rep} for the discussion in the setting of $J(\rho)$), so the functional $L^{\rm mod}_{\tau,\rho'}$ is well-defined on $J(\tau)\times \tilde{J}'(\rho')=J(\tau)\times\big(J'(\rho')/J'_0(\rho')\big)$. When $d=2$, $J'(\tau')$ is empty and this case was treated in \cite{re}.

Similar to $\phi$, the natural lift $\tilde{\psi}$
(over $\Gamma'\backslash G'$) of
the Maass form $\psi$ generates an irreducible automorphic representation $V_{\mu}'\subseteq L^2(\Gamma'\backslash G')$ of $G'$. The operator
$$L^{\rm aut}_{\tau,\tau'}:\,V_{\lambda}\times
V'_{\mu}
\rightarrow\mathbb{C},\quad (g,j)\mapsto
\int_{\Gamma'\backslash G'}g(z)j(z)dz$$
defines a nonzero $\Delta(G')$-invariant bilinear form on $V_{\lambda}\times
V'_{\mu}$.  
Denote by $\mathcal{B}'$ and 
$\mathcal{A}'$ the 
isomorphisims 
$V'_{\mu}\cong
I'(\tau')$ and $I'(\tau')
\cong J'(\tau')$
respectively.
Then there exists a nonzero scalar $b_{\tau,\tau'}\in\mathbb{C}$ such that 
\begin{equation}\label{lin}
L^{\rm aut}_{\tau,\tau'}(g,j)=b_{\tau,\tau'}\cdot L^{\rm mod}_{\tau,\tau'}(f,h)
\end{equation}
where $f=\mathcal{A}\circ
\mathcal{B}(g)$, 
$h=\mathcal{A}'\circ
\mathcal{B}'(j)$.
Note that, for fixed $\tau$, $\tau'$, the scalar $b_{\tau,\tau'}$ 
does not depend on the choice of $f$ and $h$.
The spherical element in 
$J'(\tau')$ that  corresponds to the normalized Maass  form $\psi$ is 
$\mathfrak{h}(y)=
\mathcal{A}'\circ
\mathcal{B}'(\tilde{\psi})
=(1+|y|^2)^{-(\tau'+\rho')}$. In \cite{mo}, the value of  $L^{\rm mod}_{\tau,\tau'}(\mathfrak{f},\mathfrak{h})$ was explicitly computed (see Proposition 3.1 therein):
$$L^{\rm mod}_{\tau,\tau'}(\mathfrak{f},\mathfrak{h})=\frac{\pi^{\rho+\rho'}\Gamma(\rho')\Gamma\left(\frac{\tau+\rho+\tau'-\rho'}{2}\right)\Gamma\left(\frac{\tau+\rho-\tau'-\rho'}{2}\right)}{\Gamma(2\rho')\Gamma(\rho-\rho')\Gamma(\tau+\rho)}.$$

By Stirling's asymptotic formula
$$|\Gamma(x+\textup{\bf i}\,y)|=\sqrt{2\pi}\,|y|^{x-1/2}e^{-\frac{\pi |y|}{2}}
\Big(1+\mathcal{O}\left(|y|^{-1}\right)\Big),
\quad\textup{as}~|y|\rightarrow\infty,$$
one verifies the following asymptotic
\begin{equation}\label{asymp}
L^{\rm mod}_{\tau,\tau'}(\mathfrak{f},\mathfrak{h})
\asymp |\tau|^{-\frac{d-1}{2}},
\quad \textup{as}\,~|\tau|\rightarrow\infty,
\end{equation} 
for any fixed $\tau'$. 
Taking account of (\ref{asymp}), to show Theorem \ref{thm} it suffices to bound the scalar $b_{\tau,\tau'}$.  
We shall prove
\begin{prop}\label{prop} For any $\tau'$  
there exists a positive number $c_{\tau'}$ depending on $\tau'$ such that
$$\left|b_{\tau,\tau' }\right|\leq c_{\tau'}\,|\tau|^{\frac{d-1}{2}},\quad\text{as}\,~|\tau|\rightarrow\infty.$$
\end{prop}
Given any fixed $\psi$, the condition (\ref{mult}) is fulfilled when $|\tau|$ is large (since $\tau\in \textup{\bf i}\,\mathbb{R}$). 
In view of the uniform asymptotic $|\tau|\sim\lambda^{1/2}$ (as $\lambda\rightarrow\infty$), Theorem \ref{thm} follows from  (\ref{asymp}) and Proposition \ref{prop}. To prove Proposition \ref{prop} we need two lemmas. 

Denote $a_r=\exp(\log rE)\in A$ for $r>0$. 
We have the following commutativity relations:
\begin{equation}\label{akn}
a_rn_x=n_{rx}a_r,\qquad
a_r\bar{n}_x=
\bar{n}_{r^{-1}x}
a_r,\qquad \iota_{_\flat}(k)n_x=n_{_{(xk^{\tiny\mathrm{T}})}}
\iota_{_\flat}(k),\qquad
%\iota_{_\flat}(k)\bar{n}_x=\bar{n}_{_{(xk)}}\iota_{_\flat}(k),
\end{equation}
where $k\in O_{d-1}$. See Section \ref{rep} for the meaning of $\iota_{_\flat}$.

Define   $$N^{c}=\left\{n_x\,|\,x=(x_i)_{i=1}^{d-1}\in\mathbb{R}^{d-1},~x_i=0~{\rm for}~i\leqslant d-2\right\},$$
$$\bar{N}^c=\left\{\bar{n}_x\,|\,x=(x_i)_{i=1}^{d-1}\in\mathbb{R}^{d-1},~x_i=0~{\rm for}~i\leqslant d-2\right\}.$$
Then $N=N'N^c$ and $\bar{N}=\bar{N}'\bar{N}^c$ where
$$N'=
\left\{n_x\,|\,x=(x_i)_{i=1}^{d-1}
\in\mathbb{R}^{d-1},
~x_{d-1}=0  \right\},$$
$$\bar{N}'=
\left\{\bar{n}_x\,|\,x=(x_i)_{i=1}^{d-1}
\in\mathbb{R}^{d-1},
~x_{d-1}=0  \right\}.$$
According to the Bruhat decomposition, we may write $\bar{n}_{(y_1,\ldots,y_{d-2},0)}\cdot
n_{(0,\ldots,0,\delta)}\in \bar{N}'N^c$ as
\begin{equation}\label{write}
\bar{n}_{(y_1,\ldots,y_{d-2},0)}\cdot n_{(0,\ldots,0,\delta)}=n_zma_t
\bar{n}_s\in
 NMA\bar{N}
\end{equation}
with $z=(z_i)$, $s=(s_i)\in
\mathbb{R}^{d-1}$,
$t\in\mathbb{R}_+$ and $m=\iota_{_\flat}(\rchi)$ for some $\rchi=(\rchi_{ij})\in O(d-1)$.
The computation shows 
$t=\frac{1}{1+
\delta^2|y|^2}$ and
\begin{equation}\label{7} z_i=\left\{\begin{array}{ll}
\frac{\delta^2y_i}{1+\delta^2|y|^2},&\text{if}\,~1\leqslant i\leqslant d-2,\\[0.3cm]
\frac{\delta}{1+\delta^2|y|^2},&\text{if}\,~i=d-1,\end{array}
\right.
\end{equation} 
\begin{equation}\label{8} s_i=\left\{\begin{array}{ll}
\frac{y_i}{1+\delta^2|y|^2},&\text{if}\,~1\leqslant i\leqslant d-2,\\[0.3cm]
\frac{\delta|y|^2}{1+\delta^2|y|^2},&\text{if}\,~i=d-1,\end{array}
\right.
\end{equation} 
\begin{equation}\label{9}
2z_{d-1}s_it^{-1}+
\rchi_{d-1,\,i}=
\left\{\begin{array}{ll}0,&\text{if}\,~1\leqslant i\leqslant d-2,\\[0.3cm]
1,&\text{if}\,~i=d-1.\end{array}
\right.
\end{equation} 
\begin{lem}\label{lem0}
Under the natural quotient map, the image of 
$N^c\bar{N}$ in $G'\backslash G$ is isomorphic to
$${\cal F}:=\big(N^c\bar{N}
\smallsetminus
\bar{N}\big)\cup\bar{N}^c.$$ 
\end{lem}
\begin{proof}
Assume that there exist  $n_1n_2$, $n_3n_4\in N^c\bar{N}$ and 
$g_0\in G'$
such that
$n_1n_2=
g_0n_3n_4$. Write
\begin{itemize}[leftmargin=*]
\item 
$n_1=n_{(0,\ldots,0,\beta)}$,
$n_3=n_{(0,\ldots,0,\delta)}$ where $\beta$, $\delta\in\mathbb{R}$, 
\item 
$n_2=\bar{n}_v$,
$n_4=\bar{n}_\zeta$ where $v=(v_1,\ldots,v_{d-1})$, $\zeta=(\zeta_1,\ldots,\zeta_{d-1})\in\mathbb{R}^{d-1}$,
\item 
$g_0=n_{(w_1,\ldots,w_{d-2},0)}\cdot a_{\ell}\cdot
\iota_{_\flat}(\tilde{k})\cdot\bar{n}_{(y_1,\ldots,y_{d-2},0)}\in N'AM'\bar{N}'=G'$
where $w_i$, $y_i\in\mathbb{R}$,  $\ell\in\mathbb{R}_+$ and $\tilde k=\textup{diag}(k,\omega)\in O(d-1)$ for some $k\in O(d-2)$, $\omega=\pm1$. 
\end{itemize}
By  (\ref{akn}), (\ref{7}) and (\ref{8}), the equality $n_1\bar{n}_2=g_0n_3\bar{n}_4$ reads
\begin{equation}\label{reads}
n_{(0,\ldots,0,\beta)}{\cdot}\bar{n}_v=n_{(w_1,\ldots,w_{d-2},0)+\ell\cdot z\cdot\tilde{k}^{\tiny\mathrm{T}}}\cdot\iota_{_\flat}
(\tilde{k}\cdot\rchi)\cdot a_{\ell t}\cdot\bar{n}_{s+\zeta}.\end{equation} 
Comparing the two sides of \eqref{reads} yields
\begin{equation}\label{kr}
\tilde{k}\cdot\rchi=1\in SO(d-1),
\end{equation} 
\begin{equation}\label{lt}
\ell t=1,
\end{equation} 
\begin{equation}\label{vsm}
v=s+\zeta,
\end{equation} 
\begin{equation}\label{uwl}
(0,\ldots,0,\beta)=(w_1,\ldots,w_{d-2},0)+\ell\cdot z\cdot\tilde{k}^{\tiny\mathrm{T}}.
\end{equation} 
By (\ref{kr}), the 
$(d-1)$-th row of $\rchi$ is $\rchi_{d-1,\,\ast}=(0,\ldots,0,\omega)$. So we have
$\rchi_{d-1,d-1}=\omega=\pm1$
and 
$\rchi_{d-1,i}
=0$ for $i=1,\,\ldots\,d-2$.
\begin{itemize}[leftmargin=*]
\item 
If $\rchi_{d-1,d-1}=\omega=1$, then 
$z_{d-1}s_{d-1}=0$ by (\ref{9}). It follows that 
$z_{d-1}=0$  or $s_{d-1}=0$. 
Note that $z_{d-1}=0$ is 
equivalent to $\delta=0$ by (\ref{7}). 
\begin{itemize}
\item[(1)]
If $z_{d-1}=0$, i.e., 
$\delta=0$, then by (\ref{write}) we have:
$t=1$,  $z=0$, $m=1\in M$
and $s=(y_1,\ldots,y_{d-1},0)$.
Thus $k=1\in
SO(d-2)$ by (\ref{kr}),
$\ell=1$ by (\ref{lt}),
$v_{d-1}=\zeta_{d-1}$ by (\ref{vsm}),
$w_1=\cdots=
w_{d-2}=\beta=0$ by (\ref{uwl}).
Then we deduce that $n_1=n_3=1$ and 
$g_0=\bar{n}_{(y_1,\ldots,y_{d-2},0)}\in
\bar{N}'$.
Now (\ref{reads}) reads
\begin{equation}\label{nnc}
\bar{n}_{(v_1,\,\ldots,\,v_{d-1})}=
\bar{n}_{
(y_1+\zeta_1,\,\ldots,\,y_{d-2}+\zeta_{d-2},\,
\zeta_{d-1})}.
\end{equation}
\item[(2)] 
If $s_{d-1}=0$, then 
$\delta|y|^2=0$ by (\ref{8}). This implies that  $|y|=0$ since
$\delta=0$ is equivalent to 
$z_{d-1}=0$ (we have discussed this case in above).  
Substituting $y_1=\cdots=y_{d-2}=0$ into (\ref{write}), we get $z=(0,\ldots,0,\delta)$,
$m=1\in M$,
$t=1$ and $s=0$. Thus
$k=1\in
SO(d-2)$ by (\ref{kr}), 
$\ell=1$ by (\ref{lt}),
and $w_1=\cdots=w_{d-2}=0$
by (\ref{uwl}). It follows that
$g_0=1$. Now (\ref{reads}) reads $n_1n_2=n_3n_4\in N^c\bar{N}$, from which 
we get $n_1=n_3$ and 
$n_2=n_4$.
\end{itemize}
\item
If $\rchi_{d-1,d-1}=\omega=-1$, 
then $z_{d-1}s_{d-1}t^{-1}=1$
by (\ref{9}). Hence  $z_{d-1}\ne0$ and 
$s_{d-1}\ne 0$. As
$\rchi_{d-1,i}=0$ ($1\leqslant i\leqslant d-2$),  by (\ref{9}) we have $z_{d-1}s_i=0$ for 
$i=1,\,\cdots,\,d-2$, which implies that $s_i=0$ for 
$i=1,\,\cdots,\,d-2$. Hence
$y_1=\cdots=y_{d-2}=0$
by (\ref{8}). Substituting this property into (\ref{write}) yields
$m=1\in M$, whence $\rchi_{d-1,d-1}=1$, a contradiction.
\end{itemize}
The above discussion shows that the two elements 
$n_1n_2$, $n_3n_4\in N^c\bar{N}$ have the same image in  
$G'\backslash G$ (i.e.,
 $n_1n_2=g_0n_3n_4$ for some $g_0\in G'$) in two situations:
(1) $g_0=1$, $n_1=n_3\ne1$ and $n_2=n_4$;
(2) $n_1=n_3=1$, 
$g_0\in \bar{N}'$. 
Those elements $n_1n_2$ satisfying the former situation constitute the subset $N^c\bar{N}
\smallsetminus
\bar{N}\subset
G$ whose image in $G'\backslash G$ is 
identified with $N^c\bar{N}
\smallsetminus
\bar{N}$ itself.  
Those elements $n_1n_2$ satisfying the latter situation constitute the subset $\bar{N}\subset
G$ whose image in $G'\backslash G$
is, by \eqref{nnc},  identified with $\bar{N}^c$.
\end{proof}
Denote 
\begin{gather*}
N^c_+=\{n_{(0,\ldots,0,\delta)}\cdot\bar{n}\in N^c\bar{N}\,|\,\delta>0\},\\
N^c_-=\{n_{(0,\ldots,0,\delta)}\cdot\bar{n}\in N^c\bar{N}\,|\,\delta<0\}.
\end{gather*}
Then the intuitive picture of ${\cal F}$ is just the two separate pieces
$N^c_+$, $N^c_-$
wielded together by $\bar{N}^c$. Note that ${\cal F}$ is not a 
 manifold as it has a 1-dimensional singularity  $\bar{N}^c$. 
The quotient space $G'\backslash G$ is a smooth manifold 
with dimension $d=\dim G-\dim G'$. 
Lemma \ref{lem0} essentially characterizes a special open subset of $G'\backslash G$ in terms of elements in $G$. Namely,  
the subset $N^c\bar{N}
\smallsetminus
\bar{N}$ (which is the disjoint union of $N^c_+$ and $N^c_-$), when viewed as its image in $G'\backslash G$, is an 
open subset of $G'\backslash G$ since $\dim\big(N^c\bar{N}
\smallsetminus
\bar{N}\big)=\dim
N^c_+=\dim
N^c_-=
d$. 
From now on we shall 
identify ${\cal F}$
with its image in $G'\backslash G$.

For $f\in J(\tau)$, denote by $\Phi_f$ the image of 
$f$ in $V_{\lambda}$
under the map $\mathcal{B}^{-1}
\circ\mathcal{A}^{-1}$. Let $U_{\alpha}$ be a subset of 
${\cal F}$ defined to be $U_{\alpha}=\{n_{(0,\ldots,0,\delta)}\cdot\bar{n}_{(z_1,\ldots,z_{d-1})}:0<\delta<
\alpha,\,|z_i|<\alpha\}\subset \bar{N}^c_+$
where $\alpha$ is a positive number.  Clearly, $U_{\alpha}$ 
is open and simply connected in $G'\backslash G$. The point $G'{\cdot}e\in
G'\backslash G$ lies on the boundary of $U_{\alpha}$.
We require that $U_{\alpha}$ is fixed and small enough, i.e., 
$\alpha$ is fixed and small. As an open subset of 
$G'\backslash G$, $U_{\alpha}$ is equipped with a measure
$du$ induced from the invariant Radon measure 
$d\tilde{g}$ of $G'\backslash G$ such that 
$\textup{vol}(U_{\alpha})>0$. Recall that 
$d\tilde{g}$ is uniquely decided by the 
quotient integral formula: $\int_Gw(g)dg=
\int_{G'\backslash G}\int_{G'}
w(g^{\,\prime}\tilde{g})
dg^{\,\prime}d\tilde{g}$ for any $w\in L^1(G)$. Here $dg$, $dg^{\,\prime}$ 
are fixed Haar measures on $G$ and $G'$ respectively. 
Let $\xi$ be a continuous nonnegative 
real-valued function over $U_{\alpha}$ such that $\xi$ is
 compactly supported and  $\int_{U_{\alpha}}\xi(u)du=1$. 
Let $dt$ be the invariant Radon measure on $\Gamma'\backslash G'$
induced from $dg^{\,\prime}$.
\begin{lem}\label{lem2}
For large enough $T$ such that
$T\geq|\tau|$, there exist 
$f_{_{T}}\in J(\tau)$,
$h\in J'(\tau')$ such that $\|f_{_{T}}\|^2_{J(\tau)}=
\|h\|_{J'(\tau')}^2=1$ and 
\begin{enumerate}[label=\upshape(\roman*)]
\item\label{B}
$\big|L^{\rm mod}_{\tau,\tau' }(u\bm{.}f_{_{T}},h)\big|
\geq d_{f,h}\cdot T^{-\frac{d-1}{2}}$ 
for any $u\in U_{\alpha}$,
\item\label{C}
$\int_{U_{\alpha}}\int_{\Gamma'\backslash G'}
\big|\Phi_{f_{_{T}}}(tu)\big|^2\xi(u)dtdu\leqslant b$,
\end{enumerate}
where $d_{f,h}$ is a positive number depending on $f$, $h$, 
and $b$ is a positive number depending
on $\xi$.
\end{lem}
\begin{proof}
To prove \ref{B}, we first assume that $\tau'\ne \rho'$, i.e., $J'(\tau')$ is irreducible. The case where $\tau'=\rho'$ will be discussed later. Let $f$ be a smooth nonnegative real-valued function 
over $\mathbb{R}^{d-1}$ such that 
\begin{itemize}[leftmargin=*]
\item
the support of $f$ is around 0 and sufficiently small; 
\item $\|f\|^2_{J(\tau)}=1$.
\end{itemize} 
See Section \ref{rep} for the definition of $\|\cdot\|_{J(\tau)}$. 
Define $f_{_{T}}$ to be
$$f_{_{T}}(x)=T^{\frac{d-1}{2}}\cdot f\big(Tx_1,\ldots, Tx_{d-2},
T(x_{d-1}-1)\big).$$
Clearly, $\|f_{_{T}}\|^2_{J(\tau)}=\|f\|^2_{J(\tau)}=1$ (remember that $\tau$ is purely imaginary). 
Let $h$ be a smooth nonnegative real-valued function on 
$\mathbb{R}^{d-2}$ 
such that
\begin{itemize}[leftmargin=*]
\item
the support of $h$ is around 0 and sufficiently small; 
\item $\|h\|^2_{J'(\tau')}=1$.
\end{itemize}

Prior to the detailed argument, we give a brief account on the idea for the proof of property \ref{B}. First off, we show that the kernel function $K_{\tau,\tau'}(x,y)$ is, roughly speaking, a nonzero constant for $x\in{\rm supp}(u\bm{.}f_{_T})$, $y\in{\rm supp}\,h$ (where $u$ lies in $U_\alpha$) and large $T$. Therefore $L^{\rm mod}_{\tau,\tau'}(u\bm{.}f_{_T},h)$ is of the same order with $\int_{\mathbb{R}^{d-1}}\int_{\mathbb{R}^{d-2}}u\bm{.}f_{_T}(x)h(y)dxdy$, that is, the kernel function plays no essential role in the model functional (for our choice of $u\bm{.}f_{_T}$ and $h$). Then we show that the latter integral has the same order with $\int_{\mathbb{R}^{d-1}}\int_{\mathbb{R}^{d-2}}f_{_T}(x)h(y)dxdy$.

Let $u=n_\nu\bar{n}_w\in
U_{\alpha}\subset
N^c\bar{N}$
where $\nu=(0,\ldots,0,\delta)$, $w=(w_1,\ldots,w_{d-1})\in
\mathbb{R}^{d-1}$.
By the assumption on $U_{\alpha}$,
both $\delta$ and $|w|$ are small.
Let $x=(x_1,\ldots,x_{d-1})$
and
write $n_\nu^{-1}
\bar{n}_x$ as $n_\nu^{-1}
\bar{n}_x=n_{-\nu}\bar{n}_x=
\bar{n}_zma_t
n_s\in\bar{N}MAN$ 
for some $z=(z_1,\ldots,z_{d-1})$,
$s=(s_1,\ldots,s_{d-1})\in
\mathbb{R}^{d-1}$ and $t\in\mathbb{R}_+$ (note that $z$, $t$, $s$ differ from those used in Lemma \ref{lem0}). The computation shows
\begin{equation}\label{zt-a}
t=1-2x_{d-1}\delta+|x|^2\delta^2,
\end{equation} 
\begin{equation}
\label{zt-b}
z_i=\left\{\begin{array}{ll}
\frac{x_i}{1-2x_{d-1}\delta+|x|^2\delta^2},&\textup{if}\,~1\leqslant i\leqslant d-2,\\[0.3cm]
\frac{x_{d-1}-\delta|x|^2}{1-2x_{d-1}\delta+|x|^2\delta^2},&\textup{if}\,~i=d-1.
\end{array}\right.
\end{equation} 
%\begin{equation}\label{rt-3}
%s_i=\left\{\begin{array}{ll}
%\frac{x_i|x|^2}{1+2x_{d-1}v+|x|^2v^2},&1\leqslant i\leqslant d-2\\[0.2cm]
%\frac{v+x_{d-1}|x|^2}{1+2x_{d-1}v+|x|^2v^2},&i=d-1.
%\end{array}\right.
%\end{equation} 
Thus we have
\begin{align*}
u\bm{.}f_{_{T}}(x)&=n_\nu\bm{.}\big[\bar{n}_w\bm{.}f_{_T}(x)\big]\\
&=a(n_\nu^{-1}\bar{n}_x)^{
-(\tau+\rho)}\big[\bar{n}_w
\bm{.}f_{_{T}}(n_\nu^{-1}\bm{.}\,x)\big]\\
&=t^{-(
\tau+\rho)}
\big[\bar{n}_w\bm{.}
f_{_T}(z)\big]\\
&=
t^{-(
\tau+\rho)}
f_{_{T}}(z-w)
\end{align*}

Next we shall analyze $\mathcal{T}$, $\mathcal{S}$ and $t^{-(\tau+\rho)}$ independently (recall that $\mathcal{T}\cdot\mathcal{S}$ is the kernel function).
The support of $u\bm{.}
f_{_{T}}$ is equal to 
$u\bm{.}{\rm supp}(f_{_{T}})
=n_\nu\bm{.}\big(\bar{n}_w\bm{.}\,{\rm supp}(f_{_{T}})\big)=
n_\nu\bm{.}\big(w+{\rm supp}(f_{_{T}})\big)$.
Write
$p\in{\rm supp}(f_{_{T}})$
as $p=(p_i)\in\mathbb{R}^{d-1}$.
By (\ref{zt-a}) and (\ref{zt-b}),
$x\in{\rm supp}(u\bm{.}f_{_{T}})$ can be written as $x=(x_i)\in\mathbb{R}^{d-1}$
where 
\begin{equation}\label{zt-2}
x_i=\left\{\begin{array}{ll}
\frac{p_i+w_i}{1+2\delta(p_{d-1}+
w_{d-1}) +|p+w|^2\delta^2},&1\leqslant i\leqslant d-2,\\[0.3cm]
\frac{p_{d-1}+
w_{d-1}+
|p+w|^2\delta}{1+2\delta(p_{d-1}+
w_{d-1})+
|p+w|^2
\delta^2},&i=d-1.
\end{array}\right.
\end{equation} 
Any $p=(p_i)\in{\rm supp}\,f_{_{T}}$ can be written as
$$p_1=\tfrac{\epsilon_1}{T},~~\dots,~~p_{d-2}=\tfrac{\epsilon_{d-2}}{T},~~p_{d-1}=1+\tfrac{\epsilon_{d-1}}{T},$$
where $\epsilon_j$'s are small real numbers.
Denote $\epsilon=(\epsilon_1,\ldots,\epsilon_{d-1})$.  Substituting this parameterization on 
$p$ into (\ref{zt-2}),
we get 
\begin{equation}\label{newx}
x_i=
\left\{
\begin{array}{ll}
\frac{w_i+\epsilon_iT^{-1}}{V_0(\delta,w)+V_1(\delta,w,\epsilon)T^{-1}+V_2(\delta,\epsilon)T^{-2}},&1\leqslant i\leqslant d-2,\\[0.3cm]
\frac{W_0(\delta,w)+W_1(\delta,w,\epsilon)T^{-1}+W_2(\delta,\epsilon)T^{-2}}{V_0(\delta,w)+V_1(\delta,w,\epsilon)T^{-1}+V_2(\delta,\epsilon)T^{-2}},&i=d-1,
\end{array}\right.
\end{equation}
where 
$$W_0(\delta,w)=
1+w_{d-1}+
\delta\left(|w|^2+1+2w_{d-1}\right),$$
$$W_1(\delta,w,\epsilon)=\epsilon_{d-1}
+2\delta\left(\epsilon_{d-1}+
\sum\limits_{i=1}^{d-1}
\epsilon_iw_i
\right),$$
$$W_2(\delta,\epsilon)=|\epsilon|^2\delta,$$
$$V_0(\delta,w)=1+2\delta(1+w_{d-1})
+\delta^2\left(|w|^2+1+2w_{d-1}\right),$$
$$V_1(\delta,w,\epsilon)=
2\delta\left(\epsilon_{d-1}+
\epsilon_{d-1}\delta+
\sum\limits_{i=1}^{d-1}
\epsilon_iw_i\delta
\right),$$
$$V_2(\delta,\epsilon)=|\epsilon|^2\delta^2.$$
Since $\delta$, $|w|$ and 
$|\epsilon|$ are very small, 
it follows that $W_0(\delta,w)$, $V_0(\delta,w)$ are close
to $1$, and $W_1(\delta,w,\epsilon)$, $V_1(\delta,w,\epsilon)$, 
$W_1(\delta,\epsilon)$, $V_1(\delta,\epsilon)$ are close to 0.
Denote by 
$P$, $Q$ the numerator and denominator of  $x_{d-1}$
in (\ref{newx}).
It is those $x$ lying in the support of $u\bm{.}f_{_{T}}$ 
that contribute to the integral $L^{\rm mod}_{\tau,\tau' }(u\bm{.}f_{_{T}},h)$, 
so in what follows
$x$ is always assumed to 
lie in 
${\rm supp}(u\bm{.}f_{_{T}})$.
For 
such $x$, we have $P$, $Q>0$ and
\begin{equation}\label{long}
\tau\log P=
\tau\log W_0(\delta,w)+ 
\tau\log
\left(1+\frac{
W_1(\delta,w,\epsilon)}{W_0(\delta,w)}T^{-1}+
\frac{
W_2(\delta,\epsilon)}{W_0(\delta,w)}T^{-2}\right).
\end{equation}
Since $|\tau|\leqslant T$ (a condition in the lemma),
applying the Taylor expansion
of $\log(1+a)$ 
to $$a=\frac{
W_1(\delta,w,\epsilon)}{W_0(\delta,w)}T^{-1}+
\frac{
W_2(\delta,\epsilon)}{W_0(\delta,w)}T^{-2}$$
shows that the second term on the right hand side 
of (\ref{long})
converges to 
$$W:=\frac{
W_1(\delta,w,\epsilon)}{W_0(\delta,w)}{\cdot}\lim_{T
\rightarrow\infty}|\tau|\,T^{-1},$$
as $T\rightarrow\infty$.
As $\lim_{T
\rightarrow\infty}|\tau|\,T^{-1}\leqslant 1$,  the number
$|W|$
is very small. Hence,
$e^{\tau\log P}$ is close to $e^{\tau
\log W_0(\delta,w)}$ (when $T$ is large enough), 
a complex number number which is
of norm $1$ and independent of 
$x$.  
The similar argument and conclusion 
also hold for
$e^{\tau\log Q}$. 
As a result, 
when $T$ is sufficiently large,
$|x''|^{\tau-\rho}=
|x_{d-1}|^{\tau-\rho}=e^{(\tau-\rho)
\log(P/Q)}$ is close to 
$$e^{\tau\log\frac{W_0(\delta,w)}{V_0(\delta,w)}}=
\left(\frac{W_0(\delta,w)}{V_0(\delta,w)}\right)^{\tau},$$ a complex number
which is of norm $1$ and independent of $x$. 
It is easy to see that 
$|x''|^{\tau'-\rho'}$ is close 1 when $T$ is large. In summary, the second component 
$\mathcal{T}$
in the kernel of the model functional is close to a fixed 
nonzero complex number independent of $x$, for $x\in{\rm supp}(u\bm{.}f_{_{T}})$
and $T$ large.

Next we consider the term 
$t^{-(
\tau+\rho)}$
for $x\in{\rm supp}
(u\bm{.}f_{_{T}})$.
Substituting (\ref{newx}) 
into (\ref{zt-a}),
we get
$$t=
\tfrac{F_0(\delta,w)+
\sum\limits_{i=1}^4
F_i(\delta,w,\epsilon)
T^{-i}}{H_0(\delta,w)+
\sum\limits_{i=1}^4
H_i(\delta,w,\epsilon)T^{-i}}$$
where $F_i$, $H_i$
are certain polynomials of the relevant variables 
such that $F_0(\delta,w)$, $H_0(\delta,w)$ are close to 1, and 
$F_i(\delta,w,\epsilon)$, $H_i(\delta,w,\epsilon)$ ($i\geq 1$) are 
close to 0 (for small $\delta$, $|w|$ and $|\epsilon|$). 
By use of the  argument in the last 
paragraph (i.e., applying the Taylor expansion 
$\log(1+a)$ to proper $a$), we can show that 
$t^{-(\tau+\rho)}$ 
converges to the nonzero complex number
$\left(\frac{F_0(\delta,w)}{H_0(\delta,w)}\right)^{\tau}$ as $T\rightarrow
\infty$, which is independent of $x$.

The other term $\mathcal{S}$ 
in the kernel function is easier to be handled.
By (\ref{newx}) we have
$$x_i\rightarrow\frac{w_i}{V_0(\delta,w)} ~~ (1\leqslant i\leqslant d-2)~~\textup{and}~~ x_{d-1}\rightarrow
\frac{W_0(\delta,w)}{V_0(\delta,w)},\quad\textup{as}~\,
T\rightarrow\infty.$$
As $\tau'$ is fixed, we may assume that the support of $h$ is very small (around 0) such that 
$|x'-y|^2+|x''|^2$
is close to $\frac{\sum_{i=1}^{
d-2}w_i^2+W_0^2(\delta,w)}{V_0^2(\delta,w)}$. Then $\mathcal{S}$ is close to a fixed nonzero complex number which is independent of $x$ and $y$ (for large $T$).

The above discussions show  that, for $x\in{\rm supp}(u\bm{.}f_{_T})$ and $y\in
{\rm supp}(h)$, the kernel function of the model functional and 
$t^{-(\tau+\rho)}$ are close to fixed nonzero 
complex numbers which are independent of $x$
and $y$, as $T\rightarrow\infty$. Thus we have the following
estimate
\begin{align*}
L^{\rm mod}_{\tau,\tau' }(u\bm{.}f_{_{T}},h)
&=\int_{\mathbb{R}^{d-1}}
\int_{\mathbb{R}^{d-2}}
K_{\tau,\tau'}(x,y)
t^{-(\tau+\rho)}
f_{_T}(z-w)h(y)dxdy\\[0.2cm]
&\asymp
\int_{\mathbb{R}^{d-1}}
\int_{\mathbb{R}^{d-2}}
f_{_{T}}(z-w)
h(y)dydx\\[0.2cm]
&=\int_{\mathbb{R}^{d-1}}
f_{_{T}}(z-w)dx\cdot
\int_{\mathbb{R}^{d-2}}
h(y)dy
\end{align*} 
To estimate the first integral in above, we deal with the Jacobian $\big(\frac{\partial x}{\partial z}\big)$.
In view of
$$\frac{\partial z_i}{\partial x_j}=\left\{
\begin{array}{ll}
\frac{\delta_{ij}(1-2\delta x_{d-1}+
|x|^2\delta^2)-x_i(2x_j\delta^2-
2\delta\delta_{j,d-1})}{\left(1-2x_{d-1}\delta
+|x|^2\delta^2\right)^2},&1\leqslant i\leqslant d-2,\\[0.3cm]
\frac{\delta_{d-1,j}
(1-2\delta x_{d-1}+
|x|^2\delta^2)-(x_{d-1}-\delta|x|^2)(-2\delta\delta_{j,d-1}+
2\delta^2x_j)}{\left(1-2x_{d-1}\delta+
|x|^2\delta^2\right)^2},
&i=d-1,
\end{array}
\right.$$
where $\delta_{ij}$ denotes the Kronecker symbol,
it is clear that 
$\frac{\partial z_i}{\partial x_j}$ is close to $\delta_{ij}$ uniformly for 
$x\in{\rm supp}(u\bm{.}f_{_{T}})$,
and small $\delta$. Hence
the determinant of the Jacobian matrix $\big(\frac{\partial x}{\partial z}\big)$ is close to 1 for 
small enough $\delta$, and we have \begin{equation*} 
\int_{\mathbb{R}^{d-1}}
f_{_{T}}(z-w)dx\cdot
\int_{\mathbb{R}^{d-2}}
h(y)dy\,\asymp\,\int_{\mathbb{R}^{d-1}}f_{_{T}}(x)dx\cdot
\int_{\mathbb{R}^{d-2}}
h(y)dy\,\asymp_{h}\, T^{-\frac{d-1}{2}}
\end{equation*}
which implies that 
$|L^{\rm mod}_{\tau,\tau' }(u\bm{.}f_{_{T}},h)|
\geq
d_{f,h}\cdot T^{-\frac{d-1}{2}}$ for some 
positive number 
$d_{f,h}$ depending on $f$ and $h$. 
Property \ref{B} is then verified for $\tau'\ne\rho'$. 

Now we treat the case where $\tau'=\rho'$. In this case,  we choose $h=\mathfrak{h}$. Let $f$ and $f_T$ be as before. Then 
$$L^{\rm mod}_{\tau,\rho'}(u\bm{.}f_T,\mathfrak{h})=\int_{\mathbb{R}^{d-2}}\mathfrak{h}(y)dy\cdot\int_{\mathbb{R}^{d-1}}|x''|^{\tau-\rho}u\bm{.}f_T(x)dx.$$ 
The first integral $\int_{\mathbb{R}^{d-2}}\mathfrak{h}(y)dy$ converges, noting that $\mathfrak{h}(y)=(1+|y|^2)^{-(d-2)}$ (here we may assume that $d\geq 3$ since $J'(\tau')$ is empty when $d=2$). As for the second integral, the argument is completely the same as before, and we have
$$\int_{\mathbb{R}^{d-1}}|x''|^{\tau-\rho}u\bm{.}f_T(x)dx\asymp T^{-\frac{d-1}{2}}.$$

Next we prove \ref{C}. When $\alpha$ is small, the product space $\Gamma'\backslash G'\times U_\alpha$ embeds into $\Gamma\backslash G$ via the map $(t,u)\mapsto tu$. 
The subset ${\cal X}_\alpha:=
\{tU_{\alpha}\,|\,t\in\Gamma'\backslash G'\}$ is 
open in
 $\Gamma\backslash G$.  So we have 
\begin{eqnarray*}
\int_{U_{\alpha}}\int_{\Gamma'
\backslash G'}\left|\Phi_{f_{_{T}}}(tu)\right|^2\xi(u)dtdu
&\leqslant&
\sup\limits_{g\in U_{\alpha}}\xi(g)\cdot\int_{U_{\alpha}}
\int_{\Gamma'\backslash G'}\left|\Phi_{f_{_{T}}}(tu)\right|^2dtdu
\\[0.2cm]
&=&b\cdot\int_{{\cal X}_\alpha}\left|\Phi_{f_{_{T}}}(x)\right|^2dx\\[0.2cm]
&\leqslant& b\cdot\int_{\Gamma\backslash G}\left|\Phi_{f_{_{T}}}(x)\right|^2dx\\[0.2cm]
&=&b\cdot\|f_{_{T}}\|^2_{J(\tau)}\\[0.2cm]
&=&b
\end{eqnarray*}
with $b=\sup\limits_{g\in U_{\alpha}}\xi(g)$. This proves \ref{C}.
\end{proof}
\begin{rmk}\label{fat}
\normalfont
The technique that is used to show 
property \ref{C}
is named ``fattening" in {\upshape\cite{re}}. 
Indeed, by the action of  $U_\alpha$ on $f_{_T}$ we enlarge
$\Gamma'\backslash
G'$ to a small open neighborhood of
$\Gamma'\backslash
G'\subset\Gamma\backslash G$. Intuitively, this is ``fattening" $\Gamma'\backslash
G'$.
\end{rmk} 
\begin{proof}[Proof of Proposition \ref{prop}]  
Let $\Psi_h\in L^2(\Gamma'\backslash G')$ 
be the corresponding element of $h$ under 
the map $\mathcal{A}^{\prime\,-1}
\circ\mathcal{B}^{\prime\,-1}$ (applied to $J'(\tau')$).
Then $\left\|\overline{\Psi_h}\right\|^2_{L^2(\Gamma'
\backslash G')}=
\|h\|^2_{J'(\tau')}=1$. Extend $\overline{\Psi_h}$ 
to an orthonormal basis $\{\kappa_{\ell}\}$ of $L^2(\Gamma'\backslash G')$. By Lemma \ref{lem2} we have 
\begin{align}
b&\geq
\int_{U_{\alpha}}\int_{\Gamma'\backslash G'}
\left|\Phi_{f_{_{T}}}(tu)\right|^2\xi(u)dtdu\notag\\[0.2cm]
&=\int_{U_{\alpha}}\left(\int_{\Gamma'\backslash G'}
\left|R(u)\Phi_{f_{_{T}}}(t)\right|^2dt\right)\xi(u)du\notag\\[0.2cm]
&=\int_{U_{\alpha}}\left(\int_{\Gamma'\backslash G'}
\left|\sum\limits_{\ell}\left\langle R(u)\Phi_{f_{_{T}}},\kappa_{\ell}\right\rangle_{\Gamma'\backslash G'}
\cdot\kappa_{\ell}(t)\right|^2dt\right)\xi(u)du\notag\\[0.2cm]
&=\int_{U_{\alpha}}\sum\limits_{\ell}\left|\left\langle R(u)
\Phi_{f_{_{T}}},\kappa_{\ell} \right\rangle_{\Gamma'\backslash G'}
\right|^2\xi(u)du\notag\\[0.2cm]
&\geq\int_{U_{\alpha}}\left|\left\langle R(u)\Phi_{f_{_{T}}},
\overline{\Psi_h} \right\rangle_{\Gamma'
\backslash G'}\right|^2\xi(u)du\notag\\[0.2cm]
&=\left|\left\langle R(u_0)\Phi_{f_{_{T}}},\overline{\Psi_h} \right\rangle_{\Gamma'\backslash G'}\right|^2\cdot 
\int_{U_{\alpha}}\xi(u)du
\qquad\textup{for some}~u_0\in U_\alpha\notag\\[0.2cm]
&=\left|L^{\rm aut}_{\tau,\tau' }\left(R(u_0)
\Phi_{f_{_{T}}},
\Psi_h\right)
\right|^2\notag\\[0.2cm]
&=\left|b_{\tau,\tau' }\cdot L^{\rm mod}_{\tau,\tau' }
\left(u_0.f_{_{T}},h\right)\right|^2\notag\\[0.2cm]
&\geq\left|b_{\tau,\tau' }
\right|^2\cdot d_{f,h}^2
\cdot T^{-(d-1)}\notag
\end{align}
Let $T=|\tau|$, then $$\left|b_{\tau,\tau' }\right|\leqslant b^{1/2}\,d_{f,h}^{-1}
\cdot|\tau|^{\frac{d-1}{2}}=c_{\tau'}\,|\tau|^{\frac{d-1}{2}}$$ 
with $c_{\tau'}=b^{1/2}\,
d_{f,h}^{-1}$ (see Remark \ref{rmk4}).  
This proves property \ref{C}.
\end{proof}
\begin{rmk}\label{rmk4}
\normalfont
The choice of $h$ in the proof of \ref{B} 
depends on $\psi$, or rather, $\tau'$. The reason is as follows.  
To circumvent the possible 
cancellation in the integral of $L^{\rm mod}_{\tau,\tau'}(u\bm{.}f_{_T},h)$ and then get a lower bound for the integral, we require that the first 
term $\mathcal{S}$
of the kernel function
is close to a fixed nonzero complex number independent 
of $x$, $y$.
For fixed $\tau'$, this is achieved by choosing 
a test function $h$ with sufficiently small support 
around $0$.
The larger $|\tau'|$ is, the smaller ${\rm supp}(h)$ should be.
Thus, $h$ and $d_{f,h}$ depend on $\tau'$. 
However, $\tau'$ does not 
influence the choice 
of $f$.
\end{rmk}
\begin{rmk}\label{rmk3}
\normalfont
The conjecture by Reznikov (see Section \ref{intro}) amounts to 
the following bound\,{\upshape:} $|b_{\tau,\tau'}|\ll|\tau|^{\varepsilon}$, 
i.e., $-\frac14$ in the power of the conjectured 
bound comes from the special value of the model functional. 
This motivates our conjecture for general 
situation\,{\upshape:} $|P_Y(\phi,\psi)|\ll
\lambda^{-\frac{n}{4}+\varepsilon}$
(as $\lambda\rightarrow
\infty$) where $Y$ is the 
$n$-dimensional special geodesic cycle 
of $X$. That is, $-\frac{n}{4}$ comes from the special value of the model functional (which is confirmed by \eqref{asymp} when $n=d-1$), and the scalar $b_{\tau,\tau'}$ should be bounded by $|\tau|^\varepsilon$ (which is widely open).
\end{rmk}

\section{Proof of Theorem \ref{thm2}}\label{pr2}
We will prove Theorem \ref{thm2} by slightly adjusting the test function 
$h$ appearing in the preceding section.
In particular, we prove uniform versions of Lemma \ref{lem2} and Proposition \ref{prop} for $\phi$, $\psi$ under certain restrictions.
The idea for the proof is, for the most part, the same as in Section \ref{pf} (so the proof in this section is sketchy). As before, the key is to show that $\mathcal{S}$, $\mathcal{T}$ and $t^{-(\tau+\rho)}$ are close to nonzero complex constants for $x$, $y$ lying in the support of special test functions. However, for fixed $h$, such an attempt fails when 
$|\tau'|$ is large. The reason is as follows. Given any $z=|x'-y|^2+|x''|^2\ne1$,
the absolute value of $\tau'\log z$ tends to
$\infty$, as 
$|\tau'|$ goes to
$\infty$, thus $\mathcal{S}=(|x'-y|^2+|x''|^2)^{\tau'-\rho'}$ (the first term of the kernel function) is not close to any fixed nonzero constant and our trick then does not work.
The similar problem will also happen to $\mathcal{T}$ (the second term of the kernel function) if $h$ is fixed and $|\tau'|$ is large. When $d=2$ (the case treated in \cite{re}), $\mathcal{S}$ does not occur since $J'(\tau')$ is empty, but $\tau'$ still occurs in $\mathcal{T}$. So the bound in \cite{re} is not uniform with respect to $\psi$.

To overcome the above-mentioned problem, we need to remedy the test function $h$.
Without loss of generality, let us write $\tau'=\textup{\bf i}\,r$ where $r\in\mathbb{R}_+$.
Let $f$, $f_{_T}$, $h$ be as in Section \ref{pf}. Define
$$h_r(y)=
r^{\frac{d-2}{2}}
h\left(ry\right).$$ 
Then 
$\|h_r\|^2_{J'(\tau')}
=\|h\|^2_{J'(\tau')}=1$ and Lemma \ref{lem2}\,\ref{C} still holds in the present situation. 
Any $y\in{\rm supp}(h_r)$ can be written as $y=(y_1,\ldots,y_{d-2})$ 
where $y_i=
\frac{\epsilon'_i}{r}$. 
Here the norm of $\epsilon'
:=(\epsilon'_1,\ldots,\epsilon'_{d-2})$
is sufficiently small by the assumption on $h$. 
Any $x\in{\rm supp}(u\bm{.}f_{_T})$ is expressed in \eqref{newx}. 
For the rest of this section,
we only consider those $x\in{\rm supp}(u\bm{.}f_{_T})$, 
$y\in{\rm supp}(h_r)$ which contribute to
the value of 
$L^{\rm mod}_{\tau,\tau' }(u\bm{.}f_{_{T}},
h_r)$.
For such $x$ and $y$, 
we can write $|x'-y|^2+|x''|^2$ as
$$\frac{P_0(\delta,w)+
\sum_{i=1}^4
\sum_{j=1}^2
P_{ij}(\delta,w,\epsilon,
\epsilon')
T^{-i}r^{-j}
}{
H_0(\delta,w)+
\sum_{i=1}^4
H_i(\delta,w,\epsilon)
T^{-i}}$$
where $H_i$'s
are the same as those appearing in the last section,
$P_0$ and $P_{ij}$  are certain  polynomials of the relevant variables
 (see the proof of Lemma \ref{lem2} for $\delta$, $w$, $\epsilon$).
All we need to know is that 
$P_0$, $H_0$ are close to 1, and
$P_{ij}$, $H_i$ ($i$, $j\geq1$) are close to 
0, for sufficiently small $\delta$, $|w|$,
$|\epsilon|$, $|\epsilon'|$. 
When applying the Taylor expansion of the logarithm function to $\big(|x'-y|^2+|x''|^2\big)^{\tau'}=\exp\big[\tau'\cdot\log(|x'-y|^2+|x''|^2)\big]$ and $|x''|^{\tau'}=\exp\big(\tau'\cdot\log|x''|\big)$, there will be terms like 
$$e^{\textup{\bf i}\,\frac{W_j(\delta,w,\epsilon)}{W_0(\delta,w)}\cdot\frac{r}{T}},\qquad
e^{\textup{\bf i}\,\frac{V_j(\delta,w,\epsilon)}{V_0(\delta,w)}\cdot\frac{r}{T}},
\qquad
e^{\textup{\bf i}\,\frac{H_j(\delta,w,\epsilon)}{H_0(\delta,w)}\cdot\frac{r}{T}}$$ occurring (among other terms). See Section \ref{pf} for the notations $W_j$, $V_j$. As before, we expect that these terms are close to $1$, or equivalently, 
\begin{equation}\label{expect}
\tfrac{W_j(\delta,w,\epsilon)}{W_0(\delta,w)}\cdot\tfrac{r}{T},\qquad
\tfrac{V_j(\delta,w,\epsilon)}{V_0(\delta,w)}\cdot\tfrac{r}{T},\qquad
\tfrac{H_j(\delta,w,\epsilon)}{H_0(\delta,w)}\cdot\tfrac{r}{T}
\end{equation}
are close to $0$ for small $\delta$, $|w|$ and $|\epsilon|$. Since $W_j(\delta,w,\epsilon)$, $V_j(\delta,w,\epsilon)$, $H_j(\delta,w,\epsilon)$ are close to $0$ (arbitrarily, if we require that $U_\alpha$ and the support of $f$, $h$ are sufficiently small) and $W_0(\delta,w)$, $V_0(\delta,w)$, $H_0(\delta,w)$ are close to $1$, we pose the following restriction 
\begin{equation}\label{restrict}
\tfrac{r}{T}\leq C,\quad\textup{or equivalently,}\quad\big|\tfrac{\tau'}{\tau}\big|\leq C, 
\end{equation}
under which our expectation will hold for large $r=|\tau'|$ (namely, those terms in \eqref{expect} are close to $0$). Here,
$C$ is any fixed positive number. In Section \ref{pf}, the condition \eqref{restrict} is also needed, but it is automatically satisfied since $\tau'$ is fixed there.

Under the restriction \eqref{restrict}, we can apply the trick used in the proof of Lemma \ref{lem2} 
to show that $\mathcal{S}$ is close to a fixed nonzero complex number which is
independent of $x$ and $y$, for sufficiently large $r$ and $T$.
Similarly and still under the restriction \eqref{restrict}, we can show that $\mathcal{T}$ and $t^{-(\tau+\rho)}$ are also close to fixed nonzero complex numbers independent of $x$ and $y$. Thus, $L^{\rm mod}_{\tau,\tau' }(u\bm{.}f_{_{T}},h_r)$ has the same order with $\int_{\mathbb{R}^{d-1}}\int_{\mathbb{R}^{d-2}}f_{_T}(x)h_r(y)dxdy$, which is $T^{-\frac{d-1}{2}}|\tau'|^{-\frac{(d-2)
}{2}}$, and Lemma \ref{lem2}\,\ref{B} is
modified to be
\begin{equation*}
\big|L^{\rm mod}_{\tau,\tau' }(u\bm{.}f_{_{T}},
h_r)\big|\geq c_2\cdot T^{-\frac{d-1}{2}}|\tau'|^{-\frac{(d-2)
}{2}}
\end{equation*}
for large 
$T$, $|\tau'|$ and any $u\in U_{\alpha}$.
Accordingly, the bound in Proposition \ref{prop} is modified to be
\begin{equation}\label{urb}
\left|b_{\tau,\tau' }\right|\leq c_1\cdot |\tau|^{\frac{d-1}{2}}
|\tau'|^{\frac{(d-2)
}{2}},\quad\textup{as}\,~|\tau|,~|\tau'|\rightarrow\infty.
\end{equation}
Note that the restriction \eqref{restrict} is required for the above modified versions of Lemma \ref{lem2}\,\ref{B} and Proposition \ref{prop}.

When both $|\tau|$ and $|\tau'|$ go to $\infty$, by Stirling's formula one verifies that
\begin{enumerate}[leftmargin=*,label=(\Alph*),itemsep=-.5ex,topsep=.5ex]
\item\label{l1}
if $\big||\tau|-|\tau'|\big|\leq D$ where $D$ is any fixed positive number, then
$$L^{\rm mod}_{\tau,\tau'}(\mathfrak{f},\mathfrak{h})\asymp_{_D} |\tau|^{-\rho+\frac14}\asymp |\tau'|^{-\rho+\frac14};$$
\item\label{l2}
if $\big||\tau|-|\tau'|\big|$ is not bounded and $|\tau|>|\tau'|$, then 
$$L^{\rm mod}_{\tau,\tau'}(\mathfrak{f},\mathfrak{h})\sim 
b_1\cdot |\tau|^{-\rho+\frac12}\big(|\tau|^2-|\tau'|^2\big)^{-\frac14},
$$ 
where $b_1$ is a positive constant;
\item \label{l3}
if $\big||\tau|-|\tau'|\big|$ is not bounded and $|\tau|<|\tau'|$, then $$L^{\rm mod}_{\tau,\tau'}(\mathfrak{f},\mathfrak{h})\sim b_1\cdot|\tau|^{-\rho+\frac12}\big(|\tau'|^2-|\tau|^2\big)^{-\frac14}e^{-\frac{\pi}{2}(|\tau'|-|\tau|)},$$ where $b_1$ is as in \ref{l2}.
\end{enumerate}
Here, $b_1$ depends only on $d$.
Combining \eqref{lin}, \eqref{restrict}, \eqref{urb} and \ref{l1}--\ref{l3} yields 
Theorem \ref{thm2}.

\begin{flushleft}\small
Institute of Mathematics, AMSS, Chinese Academy of Sciences, Beijing, 100190, China\\
E-mail: \texttt{fsu@amss.ac.cn}
\end{flushleft}

\small
\begin{thebibliography}{hhhh}
\bibitem[Bo]{bo}J. Bourgain, {\it Geodesic restrictions and $L^p$-estimates for eigenfunctions of Riemannian surfaces}, AMS Transl. Ser. 2, 226, Amer. Math. Soc., Providence, RI, 2009.
\bibitem[BGT]{bgt} N. Burq, P. G\'erard, N. Tzvetkov, {\it Restrictions of the Laplace-Beltrami eigenfunctions to submanifolds}, Duke Math. J. {\bf 138} (2007), no. 3, 445--486. 
\bibitem[CS]{cs} X. Chen, Ch. Sogge,  {\it On integrals of eigenfunctions over geodesics}, Proc. of AMS, 143 (2015), 151--161.
\bibitem[FJ]{fj}J. Franchi, Y. Le Jan, {\it Hyperbolic dynamics and Brownian motions}, Oxford Science Publications, 2012.
\bibitem[He]{he}D. Hejhal, {\it Sur certaines s\'eries de Dirichlet associ\'ees aux g\'eodesiques ferm\'ees d'une surface de
Riemann compacte}, C. R. Acad. Sci. Paris {\bf 294} (1982), 637--640. 
\bibitem[GGP]{ggp} W. Gan, B. H. Gross and D. Prasad,
{\it Symplectic local root numbers, central critical $L$-values, and restriction problems in the representation theory of classical groups}, Ast{\'e}risque ({\bf 346}) 2012, pp. 1--110.
\bibitem[GP]{ge} I. M. Gelfand,  M. I. Graev, I. I. Pyatetskii-Shapiro, {\it Representation Theory and Automorphic Functions}, Saunders, 1968.
\bibitem[Go]{go}A. Good, {\it Local analysis of Selberg's trace formula}, Springer LNM 1040, 1983.
\bibitem[II]{ii}A. Ichino, T. Ikeda, {\it On the periods of automorphic forms on special orthogonal groups and the Gross-Prasad conjecture}, Geom. Funct. Anal. {\bf 19} (2010), 1378--1425.
\bibitem[Kn]{kn}A. Knapp, {\it Representation theory of semisimple groups: An overview based on examples}, Princeton University Press, 2001.
\bibitem[KS]{ks}T. Kobayashi, B. Speh, {\it Symmetry breaking for representations of rank one orthogonal
groups}, Memoirs of AMS, ({\bf 238}) 1126 (2015)
\bibitem[Ma]{ma} H. Maass, {\it \"Uber eine neue Art von nichtanalytischen automorphen Funktionen und die Bestimmung Dirichletscher Reihen durch Funktionalgleichungen}, Math. Ann., ({\bf 121}) 1949, 141--183
\bibitem[M{\O}]{mo} J. M\"ollers, B. {\O}rsted, {\it Estimates for the restriction of automorphic forms on hyperbolic manifolds to compact geodesic cycles}, IMRN, 2017, no. 11, 3209--3236.
\bibitem[Re]{re} A. Reznikov, {\it A uniform bound for the geodesic period}, Forum
Math. ({\bf 3}) 27 (2013), 1569--1590.
\bibitem[Sch]{sch} J. Schwermer, {\it Geodesic cycles, arithmetic groups and their cohomology}, Bull. of AMS, 47 (2010), no. 2, 187--279.
\bibitem[Wa]{wa}J.-L. Waldspurger, {\it Sur les valeurs de certaines fonctions $L$ automorphes en leur centre symm\'etrie}, Compositio Mathematica, tome 54, no. 2 (1985), 173--242.
\bibitem[Ze]{ze} S. Zelditch, {\it Kuznecov sum formulae and Szeg\"o limit formulae on manifolds}, Communications in Partial Differential Equations ({\bf 17}) 1992, no. 1-2, 221--260. 
\bibitem[Zh]{z}W. Zhang, {\it Fourier transform and the global Gan-Gross-Prasad conjecture for unitary groups}, Ann. Math. ({\bf 180}) 2014, no. 3, 971--1049.
\end{thebibliography}
\end{document}